\documentclass[12pt,reqno,twoside]{amsart}

\usepackage{amsthm,amsmath,amssymb,mathrsfs,amsbsy}
\usepackage{graphicx,epsfig,wrapfig,sidecap,subfig}
\usepackage[all]{xy}
\usepackage[usenames]{color}
\usepackage{hyperref}
\usepackage{bm}
\usepackage[capitalize]{cleveref}

\newcommand{\diag}{{\rm diag\,}}

\newcommand{\conv}{{\rm conv}}

\xyoption{dvips}
\xyoption{color}

\newif\ifdraft\drafttrue
\draftfalse
\usepackage{color}

\title[Intersections of diagonal orbits]{Intersections of diagonal orbits}
\date{}
\author{Omri N. Solan}
\address{School of Mathematical Sciences, Tel Aviv University, Israel}
\email{omrisola@mail.tau.ac.il}

\DeclareMathOperator{\SL}{SL}

\newcommand{\calA}{{\mathcal{A}}}
\newcommand{\calB}{{\mathcal{B}}}

\newcommand{\nin}{{{\notin}}}

\newcommand{\Z}{\mathbb{Z}}
\newcommand{\R}{\mathbb{R}}

\newcommand{\supp}{{\rm supp}\,}

\newcommand{\cov}{{\rm covol\, }}

\newcommand{\WR}{{\rm WR}}
\newcommand{\ST}{{\rm ST}}

\newcommand{\be}{\mathbf{e}}

\newcommand{\rk}{\mathrm{rank\,}}
\newcommand{\invdim}{\mathrm{invdim\,}}
\newcommand{\spann}{\mathrm{span\,}}
\newcommand{\stab}{\mathrm{stab}}
\newcommand{\Imm}{\mathrm{Im}}
\newcommand{\res}{\mathrm{res}}

\newcommand{\fE}{\mathfrak{E}}
\newcommand{\fF}{\mathfrak{F}}
\newcommand{\fG}{\mathfrak{G}}

\newcommand{\fU}{\mathfrak{U}}

\newcommand {\ignore}[1]  {}

\theoremstyle{plain}
\newtheorem{thm}{Theorem}[section]
\newtheorem{lem}[thm]{Lemma}
\newtheorem{prop}[thm]{Proposition}
\newtheorem{cor}[thm]{Corollary}

\newtheorem{conj}[thm]{Conjecture}

\theoremstyle{definition}
\newtheorem{definition}[thm]{Definition}

\newtheorem{remark}[thm]{Remark}

\numberwithin{equation}{section}
\swapnumbers

\begin{document}

\maketitle
\begin{center}
	\today
\end{center}
\begin{abstract}
	Let $A\subseteq SL_n(\R)$ group of diagonal matrices with positive diagonal, let $\ST_n\subseteq X_n:=\SL_n(\R)/\SL_n(\Z)$ be the set of stable lattices, and let $\WR_n\subseteq X_n$ be the set of well-rounded lattices.
	We prove that any $A$-orbit in $X_n$ intersects both $\ST_n$ and $\WR_n$.
\end{abstract}

\section{Introduction}

Let $A\subseteq SL_n(\R)$ be the diagonal subgroup and let $X_n:=\SL_n(\R)/\SL_n(\Z)$ be the space of lattices.
It is believed that Minkowski suggested the following conjecture:
\begin{conj}
	For every $\Lambda\in X_n$ and $p\in \R^n$ there exist $a\in A$ and $v \in \Lambda$ such that $\|a(p-v)\|\le {\sqrt{n}\over 2}$. 
\end{conj}

The conjecture was proved for $n\le 9$. The first proofs for $n\le 5$ used the following strategy, known as the Remak-Davenport approach.
Define the set of well-rounded lattices $\WR_n\subseteq X_n$ as the set of all lattices such that all the Minkowsi successive minima are equal. 
The Remak-Davenport approach states that to prove Minkowski's conjecture it is enough to prove the following two statements.
\begin{itemize}
	\item [$(W_n)$]  For every lattice $\Lambda\in X_n$ we have $A\Lambda\cap \WR_n\neq \emptyset$.
	\item [$(C_n)$] For every $\Lambda\in \WR_n$,
	\[\sup_{p\in \R^n}\inf_{v\in \Lambda}\|p-v\| \le \frac{\sqrt{n}}{2}.\] 
\end{itemize}

The cases $n=2, 3, 4, 5$ were proven by Minkowski \cite{M2}, Remak \cite{R}, Dyson \cite{D}, and Skubenko \cite{S},  respectively.

McMullen \cite{Mc} proved a weaker version of $(W_n)$, that, combined with a result of Birch and Swinnerton-Dyer \cite{BS}, demonstrated that if $(C_1), (C_2),...,(C_n)$ holds then Minkowski's Conjecture holds for $n$. 
Woods \cite{W} proved $(C_n)$ for $n=6$, and 
in \cite{HMS1}, \cite{HMS2}, and \cite{KR} Hans-Gill, Kathuria, Raka, and Sehmi proved $(C_n)$ for $n=7,8,9$. In particular, the Minkowski Conjecture indeed holds for $n\le 9$. 

Regev, Shapira, and Weiss \cite{RSW} proved that $(C_n)$ is false for $n\ge 30$, and therefore the Remak-Davenport approach is bound to fail. Shapira and Weiss \cite{SW} suggested a similar approach replacing the set of well-rounded lattices by the set of stable lattices (see Definition \ref{Stable} below).

As for $(W_n)$, McMullen \cite{Mc} proved that any bounded orbit closure $\overline{A\Lambda}\subseteq X$ intersects $\WR_n$. Levin, Shapira, and Weiss \cite{LSW} proved that every closed orbit $A\Lambda\subseteq X_n$ intersects $\WR_n$.
Shapira and Weiss \cite{SW} proved that every orbit closure $\overline{A\Lambda}\subseteq X$ intersects the set of stable lattices $\ST_n$, and concluded that the analog of $(C_n)$, when replacing $\WR_n$ by $\ST_n$, implies Minkowski's Conjecture.
In this paper we prove the following result, which strengthens results in \cite{Mc}, \cite{LSW}, and \cite{SW}.
\begin{thm}\label{thm:main}
	For every $\Lambda\in X_n$ the orbit $A\Lambda$ intersects $\ST_n$ and $\WR_n$ w.r.t. any norm.
\end{thm}
The proof is inspired by \cite{Mc} and is a combination of a topological claim and some lattice geometry.
To state the topological theorem, we need the concept of invariance dimension.
Recall that $\R^n$ acts on its subsets by translations.
\begin{definition}
	The {\em invariance dimension} of a convex open set $U\subseteq \R^n$ is the dimension of its stabilizer over $\R^n$, that is, \[\invdim U:=\dim \stab_{\R^n}(U).\]
	By convention $\invdim \emptyset:=-\infty$. 
\end{definition}
The topological result that we need and that extends theorem $5.1$ in \cite{Mc} is the following.
\begin{thm}\label{thm:topo}
	Let $\fU$ be an open cover of $\R^n$. Assume that \begin{enumerate}
		\item the cover \[\{\conv U:U\in \fU\}\] is locally finite;\footnotemark
		\item for every $k\le n$ and $k$ different sets $U_1,...,U_k\in \fU$ one has \[\invdim \conv (U_1\cap U_2\cap ...\cap U_k) \le n-k.\]
	\end{enumerate} Then there are $n+1$ sets in $\fU$ with nontrivial intersection. 
\end{thm}
\footnotetext{An open cover is {\em locally finite} if every compact set intersects only finitely many cover elements.}

\section{Proof of Theorem \ref{thm:main}}
We will provide some notations, most are taken from \cite{Mc}.
Define the {\em Minkowski successive minima} of a lattice $\Lambda$ by \begin{equation*}
\lambda_i(\Lambda):=\inf\{r>0:\dim\spann\{v\in \Lambda:|v|<r\}\ge i\}. 
\end{equation*}
Let $\WR_n\subseteq X_n$ be the set of all lattices for which all Minkowski successive minima are equal. Although the standard definition of $\WR_n$ uses the euclidean norm $|\,\cdot\,|$, here we consider the analogous definition with an arbitrary fixed norm.

The Harder-Narasimhan filtration was defined in \cite{HN} and described nicely by Grayson \cite{G}. Its construction for standard lattices in $\R^n$ goes as follows.
For every discrete subgroup $\Gamma<\R^n$, denote by $\cov \Gamma$ the Euclidean volume of the group $\spann\Gamma/\Gamma$. 
By convention $\cov\{0\}:=1$. We associate to $\Gamma$ the point $p_\Gamma := (\rk(\Gamma), \log \cov \Gamma)\in \R^2$. 

For every lattice $\Lambda\in X_n$ define $S_\Lambda := \{p_\Gamma:\Gamma\le \Lambda\}$. 
Denote the extreme points of $\conv(S_\Lambda)$ by $p_0,..., p_k$, and, for each $0\le i\le k$, let $\Gamma_i \le\Gamma$ satisfy $p_i = p_{\Gamma_i}$.
A result of the Harder-Narasimhan filtration states that up to reordering, $\{0\}= \Gamma_0<...< \Lambda_k=\Lambda$, are of strictly increasing ranks. Furthermore, if $p(\Gamma)$ is an extreme point, then $\Gamma$ is the unique subgroup that is associated to this point. In addition, for every $0\le i\le k$ one has $\cov\Gamma_i\le 1$. The filtration $\{0\}= \Gamma_0<...< \Lambda_k=\Lambda$ is called the Harder-Narasimhan Filtration.
\begin{definition}\label{Stable}
	The set of {\em stable lattices} $\ST_n$ is the set of all lattices $\Lambda\in X_n$ such that the Harder-Narasimhan filtration of $\Lambda$ contains only $\{0\}$ and $\Lambda$, that is, for every $\Gamma\le \Lambda$ one has $\cov\Gamma \ge 1$. 
\end{definition}

{\bf Wedge product geometry.}
Denote $\be_1,...,\be_n$ the standard basis of $\R^n$.
A basis for $\bigwedge^k\R^n$ is given by $\be_J:=\be_{j_1}\wedge...\wedge \be_{j_k}$ for  $J=\{0<j_1<...<j_k\le n\}$. Its dual basis is denoted $\{\varphi_J:\#J=k\}\subseteq \left(\bigwedge^k\R^n\right)^*$. 

A vector in the $k$'th wedge product is called a {\em $k$-vector}.
For simplicity, for every $k$-vector $v\in \bigwedge^k\R^n$ we use the norm
\[\|\omega \|_{k-\text{vec}}:=\max_J \left|\varphi_J(\omega)\right|\] and \[\supp \omega:=\{J:\varphi_J(\omega)\neq 0\}.\]
Do not confuse the arbitrary norm $|\,\cdot\,|$ of $\R^n$ with $\|\,\cdot\,\|_{1-\text{vec}}$ on $\bigwedge^1\R^n \cong \R^n$. 

{\bf Measured subspaces.} A $k$-dimensional {\em measured subspace} is a real vector subspace $M\subseteq \R^n$ equipped with a nonzero $k$-vector $\det(M)\in \bigwedge^k M$, chosen up to sign.
We denote the set of $k$ dimensional measured subspaces by $\fG_{n,k}$. 
For every $k$ dimensional measured subspace $M$ we define $\|M\|_{MS}:=\|\det M\|_{k-\text{vec}}$. 

Any discrete subgroup $\Gamma < \R^n$ gives rise to a measured space $M(\Gamma)\in \fG_{n, \rk \Gamma}$;
the space is $\spann\Gamma$, and $\det M(\Gamma)=v_1\wedge...\wedge v_k$ for a basis $v_1,...,v_k$ of $\Gamma$.

For a vector space $V$ we define its support to be $\supp (v_1 \wedge ... \wedge v_k)$ for some (any) basis $v_i$ of $V$, this is well-defined because changing the basis only multiplies $v_1 \wedge ... \wedge v_k$ by a nonzero scalar.

An alternative definition of $\supp v$ is \[\supp v:=\{J\subseteq \{1,...,n\} \text{ of size }k:\pi_J|_v \text{ is injective}\},\]
where $\pi_J:\R^n\to \R^n$ is the projection setting all coordinates not in $J$ to 0.

{\bf Flags.}

The main object that we use in the proof is the concept of a measured flag.
A {\em measured flag} is a sequence of measured spaces $\{0=v_0<v_1<...<v_l=\R^d\}$. We impose no restrictions on the volume elements. Denote the set of measured flag by $\fF_n$ and for every measured flag $F=\{0=v_0<v_1<...<v_l=\R^d\}$ define $\|F\|_F:=\max_{l>0}\|v_l\|_{MS}$.
%
We will investigate functions $F:A\to \fF_n$ with the following properties.
\begin{definition}\label{def:good_flags}
	A function $F:A\to \fF_n$ is {\em bounded} if \[\sup_{a\in A} \|F(a)\|_{F} <\infty.\]
	
	It is {\em lower locally invariant} if for every $a\in A$ there is a neighborhood $U\subseteq A$ of the identity matrix such that $a'F(a)\subseteq F(a'a)$ for every $a'\in U$. 
	
	$F$ is {\em discrete }if the set 
	\[\{a^{-1}\det v :a\in A, v\in F(a)\}\]
	is discrete in $\bigsqcup_{k=0}^n\bigwedge^k\R^n$. 
\end{definition}

\begin{thm}\label{thm:main2}
	For any discrete bounded lower locally invariant $F$ there is a point $a\in A$ such that $F(a)$ is the trivial flag $\{0<\R^n\}$. 
\end{thm}

\begin{proof}[Proof of Theorem \ref{thm:main} using Theorem \ref{thm:main2}.]
	Fix a lattice $\Lambda_0$ and let $B(r)\subseteq \R^n$ be the ball of radius $r$ with respect to the norm $|\,\cdot\,|$.
	For ever lattice $\Lambda \in X_n$ define the Minkowski measured flag by
	\[F_\text{Mink}(\Lambda):=\{\spann B(r)\cap \Lambda:r>0\}\]
	and for every $v\in F(\Lambda)$ the volume element is given by $M(v\cap \Lambda)$.
	By Minkowsi's second theorem one can see that there is a constant $C_n$ depending only on $n$ such that $\|F(\Lambda)\|_F\le C_n$. 
	
	We will prove that for some $a\in A$ we have $a\Lambda_0\in \WR_n$. 
	We apply Theorem~\ref{thm:main2} to the flag \[F(a):=F_\text{Mink}(a\Lambda_0).\]
	By the previous discussion this flag is bounded.
	It is discrete since $\bigsqcup \wedge^k\Lambda_0$ is discrete. It is lower locally invariant by the definition of $F$. The result follows since $\WR_n=\{\Lambda\in X_n:F(\Lambda)=\{0<\R^n\}\}$.
	A similar proof, using the Harder-Narasimhan filtration instead of the Minkowski measured flag, shows that $\ST_n$ intersect every $A$ orbit.
\end{proof}

The rest of this section is dedicated to the proof of Theorem~\ref{thm:main2} Using Theorem~\ref{thm:topo}.

Denote $[a]:=\{1,...,a\}$. 
We will prove the following simple observation.
\begin{lem}\label{lem:perm}
	For every flag $F=\{0=v_0<v_1<...<v_{l}=\R^n\}$ there exist a permutation $\sigma$ of $[n]$ such that $\sigma([\dim v_i])\in \supp v_i$ for every $0\le i\le l$. 
\end{lem}
\begin{proof}
	Without loss of generality add some subspaces to the flag and assume that $l=n$, that is, all dimensions appear in $F$ and $\dim v_i=i$ for every $0\le i\le n$. Recall that for every $J\subseteq [n]$ we denoted by $\pi_J$ the projection $\R^n\to \R^{n}$ setting all coordinates not in $J$ to 0, which has rank $\#J$. 
	
	We construct the permutation $\sigma$ inductively. At the $k$'th stage we will construct $\sigma(k)$ such that $\pi_{\sigma([k])}|_{v_k}$ is a bijection 
	(for $k=0$ this assumption is vacuous).
	Suppose by induction for some $J=\sigma([k])$ we have that $\pi_{J}|_{v_k}$ is a bijection. We will show that there exist $j'\nin J$ such that $\pi_{J\cup\{j'\}}|_{v_{k+1}}$ is a bijection and define $\sigma(k+1)=j'$. 
	Since $\dim v_{k+1}>k$ there is a nontrivial vector $v\in \ker \pi_J|_{v_{k+1}}$. Since $v\in \ker \pi_J$ all its $J$ coordinates vanish. Since it is nontrivial, there is $j'$ such that the $j'$ coordinate of $v$ is nontrivial. Denote $J':=J\cup \{j'\}$. 
	Since the $j'$ coordinate of $v$ is nontrivial, $\pi_{J'}(v)\neq 0$. But $\pi_J(v)=\pi_J\circ\pi_{J'}(v)=0$ and hence $k=\dim \pi_J(v_{k+1})<\dim \pi_{J'}(v_{k+1})$. Therefore $\pi_{J'}|_{v_{k+1}}$ is a bijection, as desired. 
\end{proof}
{\bf Convex sets}
\begin{lem}
	If $\emptyset\neq U_1\subseteq U_2\subseteq \R^n$ are open convex sets then $\invdim U_1\le \invdim U_2$. 
\end{lem}
\begin{proof}
	Assume without loss of generality that $0\in U_1$. Since for every open convex set $U$ that contains 0 we have  \[\stab_{\R^n}U=\{v\in \R^n:\R v\subseteq U\},\]
	the result follows.
\end{proof}

Define 
\begin{eqnarray*}
	\exp:\R^{n-1}_0:=\left \{(x_1,...,x_n):\sum_{i=1}^n x_i=0\right \}& \to& A\\ 
	(x_1,...,x_n)& \mapsto& \diag(\exp x_1,...,\exp x_n),
\end{eqnarray*}
and $\log:A\to \R^{n-1}_0$ be the inverse function. We will identify $A$ and $\R^{n-1}_0$ using this transformation and push all the notions of convexity that are defined on $\R^{n-1}_0$ to $A$.

Since the exponential function $x\mapsto e^x$ is convex, and since maximum preserves convexity, the function $a\mapsto\|aM\|_{MS}$ is a convex function for all $M\in \fG_{n,k}$ and so is $a \mapsto \|aF\|_F$ for all $F\in \fF_n$. 

\begin{proof}
	[Proof of Theorem \ref{thm:main2} using Theorem \ref{thm:topo}.]
	Assume to the contrary that $F:A\to \fF_n$ is discrete, lower locally invariant, nowhere trivial, and bounded by $c_F>0$. Construct the following cover of $\R^{n-1}_0$.
	For every $0<k<n$ and $k$-dimensional measured space $v$ define $U_v:=\{a\in A:av\in F(a)\}$. Let $\fU$ be the collection of sets $\{U_v\}$, where $v$ ranges over all $k$-dimensional measured spaces with $0<k<n$.
	Since $F$ is nowhere trivial, we deduce that $\fU$ is a cover of $A$. 
	
	To use Theorem \ref{thm:topo} we need to prove that its Conditions (1) and (2) holds. 
	To prove that Condition (1) holds, let $\fU'$ be the collection of sets 
	$U'_v:=\{a\in A:\|av\|_{MS}\le c_F\}$. 
	Since $F$ is bounded by $c_F$, we have $U'_v\supseteq U_v$ for every measured space $v$. Consequently, $\fU'$ is a cover, and since $F$ is discrete, it is locally finite. Hence $\fU$ is locally finite as well.
	
	To prove Condition (2) we will classify intersection of elements in $\fU$. Let $U_{v_1}, U_{v_2},...,U_{v_l}$ be elements of $\fU$ that have a nontrivial intersection $V\neq\emptyset$. For all $a\in V$ we have $av_1,...,av_l\in F(a)$, and hence $v_1,...,v_l$ form a flag. Assume without loss of generality that $0<v_1 <v_2<...<v_l<\R^n$.
	By Lemma \ref{lem:perm} there exists a permutation $\sigma:[n]\to [n]$ such that $\sigma([\dim v_k])\in\supp\, v_k$. Assume without loss of generality that $\sigma$ is the identity permutation. Note that for all $1\le k\le l$, $\vec x\in \R^{n-1}_0$ one has 
	\[\varphi_{[\dim v_k]}(\exp (\vec x) v_k)=\exp (\psi_{\dim v_k}\vec x)\varphi_{[\dim v_k]}(v_k),\]
	where \begin{eqnarray*}
		\psi_m:\R^{n-1}_0&\to& \R,\\
		\vec x=(x_1,...,x_n)&\mapsto& x_1+...+x_m.
	\end{eqnarray*}
	Denote $c_k:=\left |\varphi_{[\dim v_k]}(v_k)\right |$. 
	For every $\vec x\in \log V$ one has
	\[c_F>\|F(\exp\vec x)\|_F\ge \max_{k=1}^l \|\exp(\vec x) v_k\|_{MS}\ge \max_{k=1}^l \exp (\psi_{\dim v_k}\vec x) c_k,\]
	and hence the set $\log V$ is contained in $P:=\bigcap_{k=1}^l \psi^{-1}_{k}(-\infty,\log c_F-\log c_k)$. Since the functionals $\psi_k$ are linearly independent, the set $P$ satisfies $\invdim P=n-1-l$, and hence $\invdim\conv(V)\le n-1-l$. 
	
	We proved that the conditions of Theorem \ref{thm:topo} holds, and therefore the conclusion is as well: there is a nontrivial intersection of $n$ sets of $\fU$. As shown above, this intersection corresponds to a nontrivial flag with $n$ nontrivial elements, which is a contradiction.
\end{proof} 
\section{Proof of Theorem \ref{thm:topo}}
\subsection{Sketch of proof}
The proof of Theorem \ref{thm:topo} is a modified version of the proof of Theorem 5.1 in \cite{Mc}.
The main steps of the two proofs are the following:

\begin{enumerate}
	\item We construct a complex of presheaves 
	\begin{align*}
	\mathcal F: \xymatrix{0\ar[r]^{ d} & \mathcal F^0\ar[r]^{ d}& \mathcal F^1\ar[r]^{ d} &...}
	\end{align*}
	on $\R^n$
	such that the $n$'th cohomology of $\R^n$ w.r.t. $\mathcal F$, denoted $H^n_\mathcal F(\R^n)$, is nontrivial. We select a family $\fE$ of open subsets of $\R^n$ and calculate their $\mathcal F$-cohomologies. 
	
	\item Using conditions (1) and (2) we construct for every set of the form $V:=U_1\cap U_2\cap ...\cap U_k$ a nice set $V\subseteq E(V)\in \fE$ for which the $(n-k)$ $\mathcal F$-cohomology is trivial, and such that whenever $V_1\subseteq V_2$ we have $E(V_1)\subseteq E(V_2)$. 
	
	\item We complete the proof using some cohomological algebra. We construct a \v{C}ech-deRham double complex $\calA$ using $\mathcal F$ and $\fU$. We prove exactness in the \v{C}ech direction, and conclude that the $\mathcal F$-cohomology of $\R^n$ is equal to the total cohomology of $\calA$.
	We cover $\calA$ by a double complex $\calB$, built with $E$ instead of the intersections themselves. We show that the restriction map $\calB\to \calA$ is onto on the cohomologies.
	Then we show that $\calB$ is exact in the $\mathcal F$ direction on the $n$'th level, and hence any element in the $n$ cohomology class of $\calB$ can be represented in the class that represents $E$ of intersection of $n+1$ elements. Since there is a nontrivial $n$ dimensional $\mathcal F$ cohomology class in $\R^n$ there is an nonempty intersection of $n+1$ elements of $U$.
	
	Since $\mathcal F$ is not a sheaf, some work is needed to achieve exactness. 
\end{enumerate}

The differences between the proof of Theorem \ref{thm:topo} and of McMullen are the following: \begin{itemize}
	\item McMullen uses the complex of bounded forms while we use the complex of boundedly supported forms. 
	\item For the family $\fE$ McMullen uses cylinders, while we use convex sets. 
	\item The cohomology calculation is different: McMullen calculates it directly while we use the Mayer-Vietoris sequence. 
	\item The \v{C}ech-deRham double complex is different: McMullen used direct sum of normed spaces while we use standard direct sum.
\end{itemize}
%

\subsection{Boundedly supported forms}
Denote by $B(r)\subseteq \R^n$ the open ball of radius $r$ around $0$.
For every open set $U\subseteq \R^n$ denote by $\Omega^k(U)$ the set of $k$-forms on $U$ and by $\Omega_{bs}^k(U)$ the set of $k$-forms on $U$ that vanish outside $B(r)$ for some $r>0$.
Recall the differential transformation $ d = d_k:\Omega^k(U)\to \Omega^{k+1}(U)$. 
Denote by $H^*(U)$ the $\Omega^*(U)$-cohomology group and by $H^*_{bs}(U)$ the $\Omega^*_{bs}(U)$-cohomology group. 

\begin{definition}
	For every convex open set $U\subseteq \R^n$ we define $\deg U$ as follows. If the projection of $U$ to $\R^n/\stab_{\R^n}(U)$ is bounded then $\deg U:=\invdim(U)$; otherwise, $\deg U:=-\infty$. 
\end{definition}
For example, the convex region $U_0\subseteq\R^2$ bounded by a parabola satisfies $\invdim U_0=0$ and $\deg U_0=-\infty$, and the open cylindrical neighborhood of a line $U_1\subseteq\R^3$ satisfies $\invdim U_1=\deg U_1=1$. 

\begin{lem}
	If $U\subseteq \R^n$ is an unbounded open convex set and $\invdim U=0$, then there is a functional $\varphi$ such that \begin{align}\label{eq:functional}
	\{x\in U:\varphi(x)<r\} \text{ is bounded for every } r>0.
	\end{align}
\end{lem}
\begin{proof}
	Assume without loss of generality that $0\in U$. 
	Denote by \[A=A(U):=\{x\in \R^n:\forall \lambda>0, \lambda x\in U\}\] the union of all rays from 0 that are contained in $U$. 
	Note that $A=\bigcap_{\lambda>0}\lambda U$ is the intersection of convex sets and hence convex.
	Since\footnote{$\bar U$ is the closure of $U$.} $\frac 12 \bar{U} \subseteq U$ we have $A=\bigcap_{\lambda>0}\lambda \bar U$, and hence $A$ is closed. Let $S^{n-1}$ be the $n-1$ unit sphere and denote $C=C(U):=S^{n-1}\cap A$. 
	Since \[C=\bigcap _{\lambda>0}(S^{n-1}\cap \lambda \bar{U})\] is the intersection of nonempty compact sets that decrease as $\lambda$ goes to $0$, it is nonempty.
	We argue that $0\nin \conv(C)$. Indeed if $0\in \conv(C)$ than there exist $l>0$, $v_1,...,v_l\in C$ and positive $\alpha_1,...,\alpha_l$ such that $\sum_{i=1}^l \alpha_i v_i=0$. 
	Since $U$ is convex it follows that $V:=\spann\{v_1,...,v_l\}\subseteq U$, and hence $U$ is invariant to translations by vectors in $V$, which contradicts the assumption that $\invdim U=0$. Hence, $0\nin \conv\, C$, and there exists a functional $\varphi\in (\R^n)^*$ such that $\varphi|_{C}>1$. 
	We will show that $\varphi$ satisfies Equation (\ref{eq:functional}). Otherwise, there exists $r>0$ such that the set $U':=\{x\in U:\varphi(x)<r\}$ is unbounded. In particular \begin{align}\label{eq:cont}
	\emptyset\ne C(U')\subseteq C(U)=C.
	\end{align}
	On the other hand \[C(U')\subseteq A(U')= \bigcap_{\lambda>0}\lambda U'\subseteq \bigcap_{\lambda>0}\{x\in U:\varphi(x)<r\}=\{x\in U:\varphi(x)\le 0\},\]
	which, together with Equation (\ref{eq:cont}), contradicts $\varphi|_{C}>1$.
	Therefore $U'$ is bounded, as desired.
	
\end{proof}


\begin{thm}\label{thm:calc_coho}
	For every convex open set $U\subseteq \R^n$ and every $k\ge 0$ we have 
	\[H_{bs}^k(U)\cong \left\{\begin{array}{cc}
	\R & k=\deg U, \\ 
	0 & \text{otherwise}.
	\end{array} \right.\]
\end{thm}
\begin{proof}
	We will prove the claim by induction on $\invdim U$. 
	Assume first that $\invdim U=0$. If $U$ is bounded, we have $\Omega_{bs}^*(U)=\Omega^*(U)$, and the claim holds since $U$ is convex.
	Assume now that $U$ is unbounded and define $\Omega_{o}^k(U):=\Omega^k(U)/\Omega_{bs}^k(U)$. 
	Let $\varphi$ be functional satisfying Equation (\ref{eq:functional}).
	Choose $\omega\in \Omega^k(U)$ that represents a cocycle in $H_o^k(U)$, the $\Omega_o^k$-cohomology of $U$. 
	Then there is $r>0$ such that $d\omega$ vanishes on $V:=U\cap \varphi^{-1}(r,\infty)$. Since $V$ is a nonempty convex set, $H^k(V)=\left\{\begin{array}{cc}
	\R & k=0 \\ 
	0 & \text{otherwise}
	\end{array} \right.$, and hence there exists $\varpi\in \Omega^{k-1}(V)$ such that
	\[\omega= \left \{\begin{array}{cc}
	\text{const} & {\rm if} ~k=0 \\ 
	d \varpi & \rm otherwise
	\end{array} \right. \text{in} V.\]
	One can find a $(k-1)$-form $\varpi'\in \Omega^{k-1}(U)$ that agrees with $\varpi$ on $U\cap \varphi^{-1}(r+1,\infty)$, and thus $[\omega]\in H_o^k(U)$ is either trivial, or equivalent to the constant function if $k=0$. One can see that $ H_o^0(U)\cong \R$, since the constant functions in $\Omega_o^0(U)$ generate a nontrivial class.
	By definition the following is a short exact sequence of complexes:
	\begin{align*}
	\xymatrix{0\ar[r] & \Omega^*_{bs}(U)\ar[r]& \Omega^*(U)\ar[r]&\Omega_o^*(U)\ar[r]&0.}
	\end{align*}
	By the snake lemma the following is a long exact sequence of cohomologies:
	\begin{align*}
	0\, \to \, H_{bs}^0 (U)&\, \to \,  H^0(U) \, \to \,  H^0_o(U)\, \to \, 
	\\H_{bs}^1 (U)&\, \to \,  H^1(U) \, \to \,  H^1_o(U)\, \to \, ...
	\end{align*}
	Note that the arrow $H^k(U)\to H_o^k(U)$ is an isomorphism for every $k$. For $k=0$ the two groups are isomorphic to $\R$ and the arrow is a monomorphism. For $k>0$ both are trivial.
	Therefore, all the cohomologies in the sequence $H_{bs}^*(U)$ are $0$, and the proof for the case $\invdim U=0$ is complete.
	
	For the induction step, suppose $\invdim U=k>0$. Assume without loss of generality that $U=\R^k\times U'$ for $U'\subseteq\R^{n-k}$ with $\invdim U'=0$. 
	Write $U=U_1\cup U_2$ where $U_1:=U\cap \{x_1\ge -1\}$ and $U_2:=U\cap \{x_1\le 1\}$. Denote $V:=U_1\cap U_2$. 
	Note that $\invdim U_1=\invdim U_2=\invdim V=k-1$, $\deg U_1=\deg U_2=-\infty$, and $\deg V=\deg U-1$.
	Note that 
	\begin{align*}
	\xymatrixcolsep{3.4pc}\xymatrix{0\ar[r] & 
		\Omega^*_{bs}(U)\ar[r]^-{\alpha\mapsto (\alpha,\alpha)}& 
		\Omega^*(U_1)\times\Omega^*(U_2)\ar[r]^-{(\alpha,\beta)\mapsto \alpha-\beta}&
		\Omega_o^*(V)\ar[r]
		&0}
	\end{align*}
	is a short exact sequence of complexes and by the snake lemma there is a long exact sequence of cohomologies 
	\begin{align*}
	0\, \to \, H_{bs}^0 (U)&\, \to \,  H_{bs}^0(U_1)\oplus H_{bs}^0(U_2) \, \to \,  H_{bs}^0(V)\, \to \, \\
	H_{bs}^1 (U)&\, \to \,  H_{bs}^1(U_1)\oplus H_{bs}^1(U_2) \, \to \,  H_{bs}^1(V)\, \to \, ...
	\end{align*}
	Since $H_{bs}^*(U_1)$ and $H_{bs}^*(U_2)$ vanish we conclude that $H_{bs}^l(U)\cong H_{bs}^{l-1}(V)$ for every $l\ge 1$, as desired.
\end{proof}
\subsection{Complexes}
A {\em double complex }is a collection of Abelian groups $\{C^{p,q}\}_{p,q\ge 0}$ with two maps 
\[ d:\bigoplus_{p,q\ge 0}C^{p,q}\to \bigoplus_{p,q\ge 0} C^{p,q+1},\ \ \ 
\delta:\bigoplus_{p,q\ge 0}C^{p,q}\to \bigoplus_{p,q\ge 0}C^{p+1,q},\]
defined by the restrictions 
\[ d|_{C^{p,q}}= d_{p,q}:C^{p,q}\to C^{p,q+1},\ \ \  \delta|_{C^{p,q}}=\delta_{p,q}:C^{p,q}\to C^{p+1,q},\]
which are differentials and commute:
\[\delta^2= d^2=\delta d-\delta d=0.\]
We say that the {\em degree} of $C^{p,q}$ is $p+q$ and define the {\em total complex} of $C$ by $C^r:=\bigoplus_{p+q=r} C^{p,q}$ and 
\[D:\bigoplus_{r\ge 0}C^r\to\bigoplus_{r\ge 0}C^{r+1},\]
defined by the restrictions $D|_{C^r}=D_r:C^r\to C^{r+1}$, which in turn is defined by $D_r|_{C^{p,q}}=(-1)^q\delta_{p,q}+ d_{p,q}$.
One can verify that $D^2=0$. 
The total cohomologies of the double complex are $H_C^r:=\ker D_r/\Imm D_{r-1}$. 

\begin{lem}\label{lem:push_coho}
	If $\delta$ is exact at all groups of degree $r$, then any $\alpha \in H_C^r$ has a representative $a\in C^{0r}$. 
\end{lem}
\begin{proof}
	Let $\alpha\in C^r$ for which $D\alpha=0$. We will find $\beta\in C^{r-1}$ such that $\alpha+D\beta\in C^{0,r}$.
	Assume that \begin{align}\label{eq:push}
	\alpha=\sum_{p+q=r, p\le l} \alpha^{p,q}\in \bigoplus_{{p+q=r, p\le l}}C^{p,q},
	\end{align} where $\alpha^{p,q}\in C^{p,q}$, $\alpha^{l, r-l}\neq 0$, and $l>0$. We will show that there is $\beta\in C^{r-1}$ that satisfies $\alpha+D\beta\in \bigoplus_{{p+q=r, p\le l-1}}C^{p,q}$. Iterating this process yields the desired result.
	
	Since $D\alpha=0$ and $l$ is the maximal index in the right-most term in Equation (\ref{eq:push}), we deduce that $\delta\alpha^{l,r-l}=0$. Since $\delta$ is exact, there is $\beta\in C^{l-1,r-l}$ that satisfies $(-1)^{r-l}\delta\beta+\alpha^{l,r-l}=0$. Therefore 
	\begin{align*}
	\alpha+D\beta\in \bigoplus_{{p+q=r, p\le l-1}}C^{p,q}.
	\end{align*}
	
\end{proof}
\begin{remark}\label{rem:no_need}
	The Proof of Lemma \ref{lem:push_coho} is valid as soon as $D\alpha\in C^{0,r+1}$. 
\end{remark}
Define $C^{-1, q}=\ker\delta_{0,q}$. This construction has the following meaning: one can extend $C$ to a double complex with the new cells $C^{-1, q}$. 
Note that the image of the restriction $D|_{C^{-1, q}} = d|_{C^{-1, q}}$ lies in $C^{-1, q+1}$, and hence $C^{-1,q}$ is a complex. We denote its cohomologies by $H_{C,d}$.
Note that there is an inclusion map $\xymatrix{C^{-1,r}\ar[r]^i&C^r}$ which induces a map $\xymatrix{H_{C,d}^r\ar[r]^i& H_C^r}$. 
\begin{cor}\label{cor:iso_coho}
	If $\delta$ is exact then the map $\xymatrix{H_{C,d}^r\ar[r]^i& H_C^r}$ is an isomorphism. 
\end{cor}
\begin{proof}
	By Lemma \ref{lem:push_coho} the map $H_{C,d}^r\to H_C^r$ is onto. We will show that this map is one to one. Assume $[\alpha^{0,r}]\in H_{C,d}^r$ vanishes in $H_C^r$; that is, there exists $\beta\in C^{r-1}$ such that $ D \beta=\alpha^{0,r}$. 
	By Lemma \ref{lem:push_coho} and Remark \ref{rem:no_need} there exists $\gamma\in C^{r-2}$ such that $\beta^{0,r-1}=D\gamma+\beta\in C^{0,r-1}$. 
	Thus, $\alpha^{0,r}=D\beta =D\beta^{0,r-1}= d \beta^{0,r-1}$.
	Since $D\beta^{0,r-1}=\alpha^{0,r-1}$, one has $\delta\beta^{0,r-1}=0$, and thus $\alpha^{0,r-1}$ is trivial in $H_{C,d}^r$. 
\end{proof}

\subsection{The \v{C}ech-De Rham double complex}
We will start this section by defining the \v{C}ech-De Rham double complex. Let $\fU$ be an open cover of $\R^n$ that satisfies the conditions of Theorem \ref{thm:topo}. 
Choose an arbitrary order on the set $\fU$.

Consider the following double complex: 

\[\calA^{p,q}=C^p(\fU, \Omega^q_{bs}):=\bigoplus_{J\subseteq \fU, \#J=p+1}\Omega^q_{bs}(U_J),\]
where $U_J:=\bigcap_{U\in J}U$. We think of this direct sum as a subset of the direct product, and write its elements in coordinate form. 

The differential $ d= d_{p,q}:\calA^{p,q}\to \calA^{p, q+1}$ is the one defined on forms, and the differential \begin{align*}
\delta=\delta_{p,q}:\calA^{p,q}&\to \calA^{p+1,q}\\
(\omega_J)_{\#J=p+1}&\mapsto (\omega'_{J'})_{\#J'=p+2}
\end{align*}
is the one defined by \[\omega_J\in \Omega^q_{bs}(U_J),\ \ \  \omega'_{J'}:=\sum_{U\in J'}(-1)^{[U:J']}\omega_{J'-U}\in \Omega^q_{bs}\left(U_{J'}\right),\]
where $[U:J]$ is the index of $U$ in $J$ by the order induced from $\fU$; it is 0 if $U$ is the smallest element in $J$ and $p$ if it is the largest. 
Because only finitely many $\omega_J$ are nonzero and they all have bounded support, every $\omega'_{J'}$ vanishes outside a bounded set. Since $\fU$ is locally finite, only finitely many $U_{J'}$-s intersect any bounded set, and hence only finitely many $\omega'_{J'}$-s are nonzero. $\calA^{p,q}$ is the \v{C}ech-De Rham double complex.
One can verify that $\delta^2=0$ and that $\delta$ and $ d$ commute.

One property of the \v{C}ech-De Rham double complex is that $\delta$ is exact. 

\begin{thm}
	The differential $\delta$ is exact. 
\end{thm}

\begin{proof}
	For every collection of sets $J$ and sets $U\in J, V\nin J$ denote $J+V:=J\cup \{V\}$ and $J-U:=J\setminus \{U\}$. 
	Choose a partition of unity $\{\rho_U\}_{U\in \fU}$. 
	Let \[\omega=(\omega_J)_{\#J=p+1}\in C^p\left (\fU, \Omega^q_{bs}\right )\]
	such that only finitely many $\omega_J$-s are nonzero. 
	As in \cite[Prop 8.5]{BT} we define 
	\begin{align*}
	T:C^{p}\left (\fU, \Omega^q_{bs}\right )&\to C^{p-1}\left (\fU, \Omega^q_{bs}\right )\\\omega & \mapsto(\omega'_{J})_{\#J=p}\in C^{p-1}\left (\fU, \Omega^q_{bs}\right ),
	\end{align*}
	where
	\[\omega'_{J}:=\sum_{V\nin J} (-1)^{[V:J+V]}\rho_V\omega_{J+V}.\]
	
	Because $\fU$ is locally finite, only finitely many $\omega'_J$-s are nonzero.
	
	Note that \begin{align*}
	\delta T\omega=&\,
	\delta\left (\sum_{V\in \fU\setminus J}(-1)^{[V:J+V]}\rho_V\omega_{J+V}\right )_{\#J=p}\\=&
	\left (\sum_{U\in J}(-1)^{[U:J]}\sum_{V\nin J-U }(-1)^{[V:J-U+V]}\rho_V\omega_{J-U+V}\right )_{\#J=p+1}\\=&
	\left (\sum_{V=U\in J}(-1)^{[U:J]}(-1)^{[V:J]}\rho_V\omega_{J}\right )_{\#J=p+1}\\
	&+\left (\sum_{U\in J}(-1)^{[U:J]}\sum_{V\nin J }(-1)^{[V:J-U+V]}\rho_V\omega_{J-U+V}\right )_{\#J=p+1}\\=&
	\left (\sum_{V\in J}\rho_V\omega_{J}\right )_{\#J=p+1}+
	\left (\sum_{V\nin J}\rho_V\omega_{J}\right )_{\#J=p+1}\\
	&-\left (\sum_{V\nin J}\rho_V(-1)^{[V:J+V]}\sum_{U\in J+V}(-1)^{[U:J+V]}\omega_{J-U+V}\right )_{\#J=p+1}\\=&\,
	\omega-T\delta\omega
	\end{align*}
	Therefore, if $\omega\in \ker\delta$ then $\omega=\delta T\omega$, and hence $\delta$ is exact.
\end{proof}

Note also that $\ker\delta_{0,r}$ represents forms on $\fU$-elements that agree on pairwise intersections, and hence $\ker\delta_{0,r}\cong \Omega_{bs}^r(\R^n)$. 
From Corollary \ref{cor:iso_coho} we deduce that $H_\calA^r\cong H_{bs}^r(\R^n)$. 

Define the following double complex :

\[\calB^{p,q}:=\bigoplus_{J\subseteq \fU, \#J=p+1}\Omega_{bs}^q(\conv U_J),\] and define $ d, \delta$, and $D$ as for the double complex $\calA$. Denote the direct sum of the restriction transformations by $\res:\calB^{p,q}\to \calA^{p,q}$. Since $\res$ commutes with $ d, \delta$, and $D$ it define a map $\res_*:H_\calB^r\to H_\calA^r$. 

\begin{prop}\label{prop:onto}
	The map $\res_*$ is onto.
\end{prop}
\begin{proof}
	Let $\alpha\in H_\calA^r$. Since $\delta$ is exact and by Corollary \ref{cor:iso_coho}, we have $H_\calA^r\cong H_{d,A}^r\cong H_{bs}^r(\R^d)$, 
	and therefore the class $\alpha$ corresponds to a class $[\omega]\in H_{bs}^r(\R^d)$. 
	Choosing $\beta:=\left (\omega|_{\conv U}\right )_{U\in \fU}\in \calB^{0,r}$ we get $\alpha=[\res\beta]$ and $\delta \beta= d \beta=D\beta=0$. In particular, $\alpha\in \rm{Im}\,\res_*$. 
\end{proof}
\begin{prop}
	At the groups $\calB^{p,q}$ of degree $n$ the differential $ d$ is exact. 
\end{prop}
\begin{proof}
	It is enough to show that if $p+q=n$ and $J\subseteq \fU$ is of size $p+1$, then $H^q_{bs}(\conv U_J)=0$.
	By Theorem \ref{thm:calc_coho}, the only nontrivial cohomology of $\conv U_J$ may be at rank $\invdim \conv U_J$, and by the assumptions of Theorem~\ref{thm:topo} $\invdim \conv U_J\le n-(p+1)=q-1$. 
	Thus the $q$ boundedly supported cohomology of $\conv U_J$ is trivial, as desired.
\end{proof}

\begin{proof}[Proof of Theorem \ref{thm:topo}]
	Since $\deg \R^n=n$ it follows that $H_{bs}^n(\R^n)\cong \R\not\cong 0$. Since $H_{bs}^n(\R^n)\cong H_\calA^n$ we deduce that $H_\calA^n\not\cong 0$. Since $\res_*$ is onto it follows that $H_\calB^n\not\cong 0$. By Lemma \ref{lem:push_coho} and the exactness of $ d$ at the groups $\calB^{p,q}$ of degree $n$, we get that $\calB^{n,0}\not\cong 0$. Thus, for some $J\subseteq \fU$ of size $n+1$ the set $U_J$ is nonempty. 
\end{proof}
\section*{Acknowledgment}
The author wishes to thank Barak Weiss for presenting the problem, his guidance, and his comments on earlier versions of the paper. The author also thanks Roman Karasev and Lev Radzivilovsky for helpful discussions.
The research was partially supported by the ERC starter grant DLGAPS 279893.

\end{document}